\documentclass{amsart}

\usepackage{amsbsy,amsfonts,amsmath,amssymb,amscd,amsthm}
\usepackage{mathrsfs}
\usepackage[mathcal]{euscript}
\usepackage{enumitem}
\usepackage{graphicx}
\usepackage{epstopdf}
\usepackage{csquotes}
\usepackage{pxfonts}
\usepackage[colorlinks=true,linkcolor=blue,citecolor=green,dvipdfmx]{hyperref}
\hypersetup{linktocpage}

\theoremstyle{plain}
\newtheorem{mainthm}{Theorem}
\newtheorem{thm}{Theorem}[section]
\newtheorem{lem}[thm]{Lemma}
\newtheorem{question}{Question}
\newtheorem{claim}{Claim}

\theoremstyle{definition}
\newtheorem{definition}[thm]{Definition}

\theoremstyle{remark}
\newtheorem{rem}[thm]{Remark}

\setlist[enumerate,1]{label=(\arabic*)}
\setlist[enumerate,2]{label=(\alph*)}

\numberwithin{equation}{section}

\begin{document}

\title[Historic behavior in homoclinic classes]{Historic behavior in non-hyperbolic homoclinic classes}



\author[Barrientos et al.]{Pablo G. Barrientos}
\address[Barrientos]{Institute of Mathematics and Statistics, University Federal Fluminense-UFF, Gragoata Campus, Rua Prof.\ Marcos Waldemar de Freitas Reis, S/n-Sao Domingos, Niteroi - RJ, 24210-201, Brazil}
\email{pgbarrientos@id.uff.br}

\author[]{Shin Kiriki}
\address[Kiriki]{Department of Mathematics, Tokai University, 4-1-1 Kitakaname, Hiratuka, Kanagawa, 259-1292, Japan}
\email{kiriki@tokai-u.jp}

\author[]{Yushi Nakano}
\address[Nakano]{Department of Mathematics, Tokai University, 4-1-1 Kitakaname, Hiratuka, Kanagawa, 259-1292, Japan}
\email{yushi.nakano@tsc.u-tokai.ac.jp}

\author[]{Artem Raibekas}
\address[Raibekas]{
Institute of Mathematics and Statistics, University Federal Fluminense-UFF, Gragoata Campus, Rua Prof.\ Marcos Waldemar de Freitas Reis, S/n-Sao Domingos, Niteroi - RJ, 24210-201, Brazil}
\email{artem@mat.uff.br}

\author[]{Teruhiko Soma}
\address[Soma]{Department of Mathematical Sciences,
Tokyo Metropolitan University,
Minami-Ohsawa 1-1, Hachioji, Tokyo 192-0397, Japan}
\email{tsoma@tmu.ac.jp}

\thanks{
Kiriki and Nakano are grateful to Yongluo Cao with his colleagues
and students in Soochow University
for their hospitality and support.
This work was partially supported by JSPS KAKENHI Grants Nos.~17K05283, 18K03376.}


\date{}

\dedicatory{Dedicated to Hiroshi Kokubu for his $60^{th}$ birthday}

\commby{}

\begin{abstract}
We show that $C^1$-generically for diffeomorphisms of manifolds of
dimension $d\geq3$, a homoclinic class containing saddles of
different indices has a residual subset where the orbit of any
point has historic behavior.
\end{abstract}
\maketitle

\section{Introduction}
The purpose of this paper is to explore historic behavior in  non-hyperbolic invariant sets of  higher-dimensional diffeomorphisms. Let us begin
by explaining what historic behavior is.
For a given continuous map  $f$ on a manifold $M$, $x\in M$ and
a continuous function $\varphi:M\to \mathbb{R}$,
we consider the sequence
of  partial averages
\[
   \frac{1}{n}\sum_{i=0}^{n-1} \varphi (f^i(x))
\]
 along the forward orbit  $O^{+}(x, f)=\{x, f(x),f^2(x),\dots \}$.
If the limit of the above averages exists
as $n\to \infty$,
it is called the \emph{Birkhoff average} associated with $(x,\varphi)$.
Otherwise, such a phenomenon, which was observed in Bowen's example \cite{G92, T94},
is called historic by Ruelle \cite{R01}.
To be more precise,  we say that $O^{+}(x, f)$  has \emph{historic behavior}
if there is a $\varphi\in C^{0}(M, \mathbb{R})$ such that the
Birkhoff average associated with $(x,\varphi)$ does not exist, see \cite{T08}.

The motivation behind the study of historic behavior is the following.
The Birkhoff ergodic theorem implies that
if $\mu$ is a probability measure on $M$ which is \emph{$f$-invariant}, that is,
$\mu(B)=\mu(f^{-1}(B))$ for every measurable set $B\subset M$,
the Birkhoff average
exists for $\mu$-almost every point $x\in M$.
Hence the set of initial points whose forward orbits have historic behavior
is of $\mu$-measure zero. 
However since the Lebesgue 
measure is in general not $f$-invariant, 
the set is not always of Lebesgue measure zero. 
Furthermore, even if the set has  Lebesgue measure zero, 
it is possible for the set to be topologically large. 

In the case of uniformly hyperbolic systems, Takens \cite{T08} proves that
the existence of a Markov partition
induces historic behavior in a residual set.
However, his method cannot be applied to a
non-hyperbolic invariant set, where the Markov partition is not guaranteed.
One of causes of the lack of hyperbolicity is existence of a homoclinic tangency, i.e., a non-transverse intersection between stable and unstable manifolds.
It is shown in \cite{KS17} that 
there exist some persistent classes of surface diffeomorphisms 
with homoclinic tangencies and contracting non-trivial
wandering domains in which the orbit of every point has historic behavior, 
see below for further information.  
This result can be 
extended to certain 3-dimensional flows with heteroclinic cycles of periodic solutions \cite{LR17}.
Meanwhile, for the geometric Lorenz attractor
it is known that
there is a residual subset in the trapping region with historic behavior
\cite{KLS16}.
\smallskip

In this paper we consider invariant sets called \emph{homoclinic classes},
which form the basic pieces of a dynamical system.
Properly speaking, for a diffeomorphism with a saddle periodic point $P$, the \emph{homoclinic class} $H(P)$ is defined as
the closure of transverse intersections
of the unstable and stable manifolds of the orbit of $P$.
Note that
every maximal invariant transitive hyperbolic set is a homoclinic class but the converse is not always true.
In fact,
a homoclinic class may be non-hyperbolic
if it contains a homoclinic tangency
or as another possibility has periodic points with different indices (the dimension of the
stable manifold). This second situation is related to the notion of a heterodimensional cycle (see in Section \ref{S4} or \cite[\S 6]{BDV05}), an important topic
in non-hyperbolic dynamical systems, and will be the one studied in this article.

From now on, let $M$ be a compact connected manifold without boundary of dimension $d\geq 3$.
Let us introduce the following terminology. A homoclinic class $H(P)$ 
of a diffeomorphism $f$
has
\emph{residual historic behavior} if there is a residual subset
$R$ of $H(P)$ and there exists a $\varphi\in C^{0}(M, \mathbb{R})$ such
that the Birkhoff average associated with $(x,\varphi)$ does not exist for every $x\in R$.
The main result of the paper is the following:

\begin{mainthm} \label{thm-main} For any $C^1$-generic diffeomorphism of $M$,
 a homoclinic class containing saddles of
different indices
has residual historic behavior.
\end{mainthm}

We remind that a property holds \emph{$C^1$-generically}
if there is a
residual subset of the space of all $C^1$-diffeomorphisms satisfying this property.
Also observe that the above definition of residual historic behavior is somewhat stronger than
the initial one because the continuous function
$\varphi$ is independent of the point $x$.

Let us continue with two questions related to Theorem \ref{thm-main}.
Entropy and dimension of sets with historic behavior
was studied in shifts with the specification property in \cite{BS00} (see also \cite{BLV14, BLV18}
 where the historic set is called an irregular set).
In particular,
the set of initial points in a uniformly hyperbolic set
whose orbits have historic behavior
carries full topological entropy and full Hausdorff dimension. So
it is natural to ask whether similar properties hold for
homoclinic classes in our setting.

\begin{question}
Does the historic set of a generic homoclinic class have full topological entropy
or full Hausdorff dimension?
\end{question}

The other question is a version of Takens' last problem \cite{T08}:
\textit{are there persistent classes of smooth 
dynamical systems for which 
the set of initial states 
which give rise to orbits 
with historic
behavior has positive Lebesgue measure?
}
An affirmative answer to this problem is already given in dimension 2  as follows. 
We say that  
an open set $D$ is a 
 \emph{contracting non-trivial wandering domain} for a diffeomorphism $f$ 
if
\begin{enumerate}[label={\alph*})]
\item $f^{i}(D)\cap f^{j}(D)=\emptyset$ if $i\neq j$,
\item the union of $\omega(x)$ for any $x\in D$ is not equal to a single periodic orbit, 
\item the diameter of $f^{i}(D)$ converges to zero if $i\to \infty$.
\end{enumerate}
As mentioned previously, it is proven in \cite{KS17} 
that 
every two-dimensional diffeomorphism in any Newhouse open set
(i.e., 
an open set of $C^2$ diffeomorphisms which have 
persistent tangencies associated with some basic sets, 
see \cite[\S 6.1]{PT-book})
is contained in the $C^{r}$-closure  ($2\leq r<\infty$) 
of diffeomorphisms having contracting non-trivial wandering domains,  
where the orbit of any point has historic behavior. 
According to a conjecture of Palis, see in \cite[\S 5.5]{BDV05},  
the cause for lack of hyperbolicity besides homoclinic tangencies 
is the existence of heterodimensional cycles.
Thus the question is as follows.
\begin{question}
Does there exist a persistent class
near every diffeomorphism having a heterodimensional cycle 
in which
any diffeomorphism has a non-trivial wandering domain $D$ such that the orbit of every point in $D$ has historic behavior? 
\end{question}
\noindent
Note that \cite{KNS17} gives a condition ensuring that
3-dimensional diffeomorphisms with heterodimensional cycles
are $C^{1}$-approximated by diffeomorphisms having contracting non-trivial wandering domains
along some attracting invariant circles. 
However, these domains contain no points whose orbits have historic behavior
due to the Denjoy-like construction used in \cite{KNS17}.
\smallskip 

We close the introduction by explaining the structure of the paper
and how to obtain Theorem \ref{thm-main} proven in Section~\ref{S4}. It will be a consequence
of our Theorem~\ref{thm-main-2} in Section~\ref{S3} and a key
result of~\cite{BD12} on the generation of special type of
hyperbolic sets, called \emph{blenders}. These sets appear inside homoclinic
classes with index variation (containing periodic points of different 
indices) for generic diffeomorphisms, as is explained in  Section~\ref{S4}.  Thus in Section
\ref{S2} we recall the definitions of blenders according to the
various levels (Definitions \ref{def:blender},
\ref{def:dynamical}), all of which can be
realized by using the covering property (Definition \ref{d2.4}).
 Then in Section \ref{S3} we give a key result (Theorem \ref{thm-main-2}) of this paper, which
contains the essential ingredients for obtaining residual historic behavior associated with
dynamical blenders.

\section{Blenders}\label{S2}
Blenders (with codimension one) were initially defined by Bonnatti and D\'{\i}az
in ~\cite{BD96} and were used to construct
robustly transitive non-hyperbolic diffeomorphisms  (see also \cite{BDV05, BD12}).
Blenders with larger codimension
were studied in~\cite{NP12,BKR14,BR17}. We will use the
following definitions coming from \cite{BBD16}.

\begin{definition}[$cu$-blender] \label{def:blender} 
Let $f$ be a diffeomorphism
of a manifold $M$. A compact set $\Gamma \subset M$ is
a \emph{$cu$-blender of codimension $c>0$} if
\begin{enumerate}
\item $\Gamma$ is a transitive maximal $f$-invariant hyperbolic
set in a relatively compact  open set $U$,
\[
    \Gamma = \bigcap_{n\in\mathbb{Z}} f^n(\overline{U}) \quad
    \text{and} \quad
    T_\Gamma M= 
    E^{s} \oplus E^c \oplus E^{uu},
\]
where $E^u=E^{c}\oplus E^{uu}$ is the unstable bundle,
$uu=\dim E^{uu}\geq 1$ and $c=\dim E^{c}\geq 1$, 
\item  \label{def:blender2}
there exists an
open set
 $\mathcal{D}$ of $C^1$-embeddings of
$uu$-dimensional discs,
and
\item there exists a $C^1$-neighborhood $\mathscr{U}$ of $f$
such that, for all $D\in \mathcal{D}$ and $g\in \mathscr{U}$,
\[
  W^s_{\rm loc}(\Gamma_g) \cap D \not = \emptyset,
\]
where $\Gamma_g$ is the continuation of $\Gamma$ for $g$.
\end{enumerate}
\end{definition}
\noindent 
See in \cite[Sec.\!~3.1]{BBD16} for 
the topology of the set of such embeddings in \ref{def:blender2}.
The set
$\mathcal{D}$ is called the \emph{superposition region} of the
blender.
A \emph{$cs$-blender of codimension $c$} of $f$ is a $cu$-blender of
codimension $c$ for $f^{-1}$.
The term of \enquote{codimension} above comes from the bifurcation theory as it is explained
in~\cite{BR19}.

%
%

\subsection{Dynamical blender}
In~\cite{BBD16}, the authors introduce the notion of a strictly
invariant family of discs as a criterion to obtain a blender,
which we explain in what follows.

Let $\mathscr{D}(M)$ be the set of $uu$-dimensional (closed) discs
$C^1$-embedded into $M$  and endowed with the $C^1$-topology.
\begin{definition}[$f$-invariant family, see Def.\! 3.7 in ~\cite{BBD16}]
A family $\mathcal{D}$ of discs in $\mathscr{D}(M)$ is said to be \emph{strictly
$f$-invariant} if there exists a neighborhood $\mathscr{N}$ of
$\mathcal{D}$ in $\mathscr{D}(M)$ such that for every disc $D_0\in
\mathscr{N}$ there is a disc $D_1 \in \mathcal{D}$ with
$D_1\subset f(D_0)$.
\end{definition}
Suppose that $\Gamma$ is a transitive hyperbolic set
having a partially hyperbolic splitting $T_\Gamma M= E^s\oplus
E^c\oplus E^{uu}$ with $E^u=E^c\oplus E^{uu}$ and $uu=\dim
E^{uu}$. Moreover, assume that there exists a strictly $f$-invariant
family of $uu$-dimensional $C^1$-discs tangent to an invariant
expanding cone-field $\mathcal{C}^{uu}$ around $E^{uu}$ (see the precise definition in \cite[Sec.\!~ 3.2]{BBD16}). Then
$\Gamma$ is actually a $cu$-blender of codimension $c=\dim E^c>0$,
see~\cite[Lem.\!~3.14]{BBD16}. Motivated by this result, they introduced the following class of blenders:

\begin{definition}[Dynamical $cu$-blender] \label{def:dynamical}
Let $f$ be a diffeomorphism of a manifold $M$. A compact set
$\Gamma \subset M$ is a \emph{dynamical $cu$-blender of
codimension $c>0$} if
\begin{enumerate}
\item $\Gamma$ is a transitive maximal $f$-invariant hyperbolic
set in a relatively compact  open set $U$,
\[
    \Gamma = \bigcap_{n\in\mathbb{Z}} f^n(\overline{U}) \quad
    \text{and} \quad
    \text{$T_\Gamma M= 
    E^{s} \oplus E^c \oplus E^{uu}$}
\]
where $E^u=E^{c}\oplus E^{uu}$ is the unstable bundle,
$uu=\dim E^{uu}\geq 1$ and $c=\dim E^{c}\geq 1$, 
\item there is a strictly $Df$-invariant cone-field $\mathcal{C}^{uu}$
around $E^{uu}$ which can be extended to $\overline{U}$,
\item  there is a strictly $f$-invariant family $\mathcal{D}$
of $uu$-dimensional discs in $\mathscr{D}(M)$ such that
every disc $D$ in a neighborhood of $\mathcal{D}$ is contained in
$U$ and tangent to $\mathcal{C}^{uu}$, i.e.,
\[
  \text{$D\subset U$ \ \ and \ \
  $T_xD \subset \mathcal{C}^{uu}(x)$ \ \ for all $x\in D$}.
\]
\end{enumerate}
\end{definition}
A \emph{dynamical $cs$-blender of codimension $c$} for $f$ is a dynamical
$cu$-blender of codimension $c$ for $f^{-1}$.

\subsection{Covering property} Although
Definition~\ref{def:dynamical} is very useful, the difficulty is
 in showing the existence of a strictly invariant family of
discs. The following covering criterion
helps us to conclude when
a hyperbolic set is a dynamical blender. This criterion appeared 
in~\cite{BD96, BDV05} in the case of codimension $c=1$ and it was
extended for large codimension in~\cite{NP12,BKR14,BR17,ACW}.

Let $\Gamma$ be a horseshoe, i.e., a locally maximal invariant
hyperbolic set of a diffeomorphism $f:M\to M$ conjugated with a
full shift.  Assume that $f$ restricted to $\Gamma$ has a partially
hyperbolic splitting $T_\Gamma M= E^s\oplus E^{cu}\oplus E^{uu}$
such that
$E^{u}=E^{cu}\oplus E^{uu}$ is the unstable bundle,
$c=\dim E^{cu}\geq 1$, $uu=\dim E^{uu}\geq 1$,
and there are positive constants $\hat{\gamma}, \hat{\nu}\in \mathbb{R}$ such that
\[
  \|Df|_{E^s}\| < 1 <\|Df|_{E^{cu}}\|\leq
   \hat{\gamma}^{-1}<\hat{\nu}^{-1}<m(Df|_{E^{uu}}).
\]
Here $\|F\|$ stands for the operator norm 
and $m(F):=\|F^{-1}\|^{-1}$ for a given invertible linear map $F$. 
Moreover, we also assume that $\Gamma$ is contained in a chart of
$M$
which is in local coordinates
$\mathbb{R}^d=\mathbb{R}^{s+uu}\times \mathbb{R}^{c}$. 

For what
follows, let us define the sets of vertical and horizontal rectangles.
A \emph{vertical} rectangle $V$ on $[-1,1]^{s+uu}$ is a set of the form
$$V=\bigcup_{y\in[-1,1]^{uu}} I_y\times \{y\},$$
where for each $y$, $I_y$ is a product of $\{uu\}$-tuple closed intervals which depend
$C^1$-continuously on $y$. A \emph{horizontal} rectangle $H$ is defined in an analogous manner,
the union now being taken over $\{x\}\times I_x, \ x\in[-1,1]^{s} $.

The covering
property consists of the conditions we describe next. See Figure \ref{fig1}.
\begin{definition}[Covering property] \label{d2.4}
There are open sets $B, B_1,\dots,B_k \subset (-1,1)^c$,
horizontal and vertical rectangles $H_1,\dots,H_k$ and $V_1,\dots,
V_k$ respectively in $[-1,1]^{s+uu}$,  satisfying the following:
\begin{enumerate}
\item \label{1} $\Gamma$ is the maximal invariant set in $V\times [-1,1]^c$
where $V=V_1\cup \dots \cup V_k$.
\item \label{2} $H_i \times \overline{B_i} \subset f^{-1}(V_i\times B)$ and
$f(H_i\times B_i)$ is a vertical rectangle in $V_i\times B$.
The vertical rectangle here is defined similarly as above having the expression 
$$\bigcup_{y\in[-1,1]^{uu}} I_y\times \{y\} \times J_y.$$
\item \label{3} $\overline{B}\subset B_1\cup \dots \cup B_k$ with Lebesgue
number $L>0$.
\item \label{4}
The local strong unstable manifolds are $C^1$-embedded graphs in $[-1,1]^d$
of the form
$$D(y)=(h_s(y) ,y, h_c(y)), \ y\in[-1,1]^{uu},$$
and having the Lipschitz constant $\hat C$ of $h_c$ satisfying $\hat C < L$.
\end{enumerate}
\end{definition}
\begin{figure}[hbtp]
\centering
\includegraphics[width=125mm]{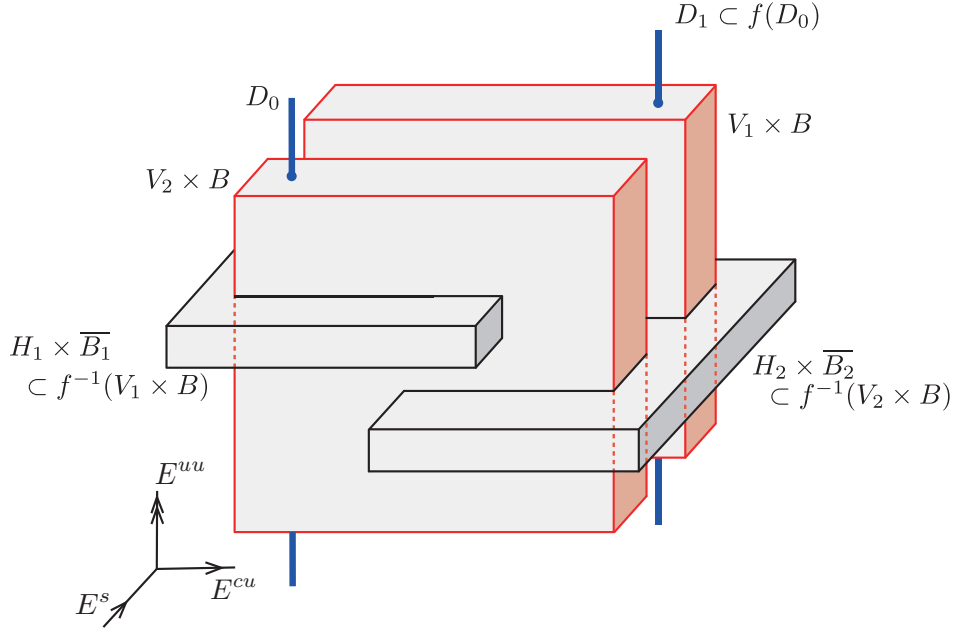}
\caption{Covering property as $k=2$}
\label{fig1}
\end{figure}

Notice that condition~\ref{4} relates the variation of the cone-field $C^{uu}$ to
the Lebesgue number of the cover in condition~\ref{3}.
Moreover,~\ref{4} always
holds when $\Gamma$ is a standard affine horseshoe
(see~\cite[Def.\!~7.4]{ACW}) or $f$ is a one-step skew-product
(see~\cite[Def.\!~5.1]{BR17}),  since in these examples $\hat{C}=0$.
In both cases the dynamics of $f$ is associated with an iterated
function system (IFS for short) generated by contracting maps
$\phi_1,\dots,\phi_k$ of $[-1,1]^c$,  and then
the fourth condition above can be summarized to asking that
\begin{equation} \label{cover}
   \overline{B} \subset \phi_1(B) \cup \dots \cup \phi_k(B).
\end{equation}


\begin{rem} \label{rem:covering-dynamical} It follows
from~\cite[proof of Thm.\!~A.2]{BR17} that a horseshoe satisfying the
covering property is a dynamical $cu$-blender. 
Namely, the set
$\mathcal{D}$ of strictly $f$-invariant  discs is given by
the set of almost-vertical $C^1$-embedded discs in $V\times B$. These \textit{almost-vertical discs} are defined
to be close
to the so-called constant vertical
discs, which are $uu$-dimensional discs projecting
into a single point  on $\mathbb{R}^c$. For more details and the precise definitions see~\cite{BR17, BR19}.
\end{rem}
\noindent



\subsubsection{Prototypical blender-horseshoe} \label{prototypical}
Any known example of a horseshoe
which is a blender
satisfies the
covering property.
This is the case of the
important model used in many
articles called the
\emph{prototypical blender-horseshoe} (see ~\cite[Sec.\!~5.1]{BD12} and also~\cite{BD96,BDV05,BD08,BDK12}).

In this model, $f$ is locally
defined as a one-step skew-product of the form $f=F\ltimes
(g_1,g_2)$ on $\mathbb{R}^{d}=\mathbb{R}^{s+uu}\times \mathbb{R}$.
Here, $F$ is a map having a horseshoe, $\Lambda$,  in $[-1,1]^{s+uu}$
conjugated with a full shift on two symbols, while
\[
g_1(x)=\lambda x,
\qquad \text{and} \qquad
g_2(x)=\lambda x -\mu
\]
for $x\in \mathbb{R}$
with
$1<\lambda<2$ and $0<\mu<1$.
Namely,
\begin{equation} \label{eq:skew-product}
  f(\omega,x)= \begin{cases}
  (F(\omega), g_1(x)) & \text{if $\omega \in H_1$} \\
  (F(\omega), g_2(x)) & \text{if $\omega \in H_2$}
  \end{cases}
\end{equation}
where $H_1$ and $H_2$ are the horizontal rectangles in $[-1,1]^{s+uu}$ 
containing $\Lambda$.
The associated contracting IFS is given by the inverse maps of the expanding $g_1$ and $g_2$.
Then, for any $\varepsilon>0$ small enough, the open set
$B=(\varepsilon, (\mu+1)/\lambda-\varepsilon)$ satisfies the
covering property \eqref{cover}. See Figure \ref{fig2}. 
\begin{figure}[hbt]
\centering
\includegraphics[width=65mm]{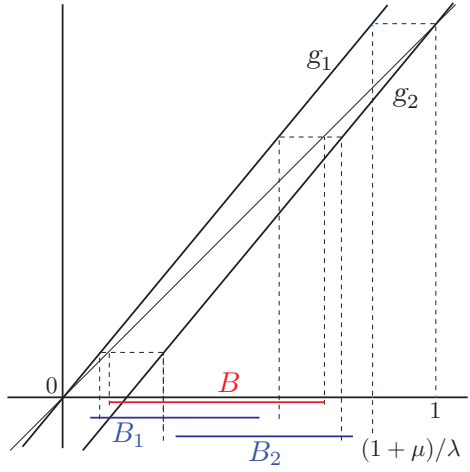}
\caption{Covering property for $(g_1,g_2)$}
\label{fig2}
\end{figure}
Hence, according to
Remark~\ref{rem:covering-dynamical}, the maximal invariant set
$\Gamma$ of $f$ in $H\times [-1,1]$, where $H=H_1\cup H_2$ 
is a dynamical $cu$-blender of
codimension $c=1$.

In what follows we will need to consider dynamical blenders having
an extra assumption. For this reason we introduce the following
definition:

\begin{definition} A periodic point $P$ in a dynamical $cu$-blender $\Gamma$ of a
diffeomorphism $f$ is said to be \emph{distal} if the orbit of $P$
is far from the strictly $f$-invariant family of disc and whose
unstable manifold contains a disc in this family.
\end{definition}

The following lemma shows that prototypical blender-horseshoes
have a distal periodic point. Moreover, since this property is
open, it also holds by any nearby dynamical blender.

\begin{lem}\label{lem-prototipical} The dynamical $cu$-blender $\Gamma$ of the map
$f$ in~\eqref{eq:skew-product} has a distal point.
\end{lem}
\begin{proof}
Let $\mathcal{D}$ be the strictly $f$-invariant family
 of $uu$-dimensional discs of $\Gamma$. According to
Remark~\ref{rem:covering-dynamical}, this corresponds with the
region $V\times B$ in $[-1,1]^d$ where $V=F(H)\cap [-1,1]^{s+uu}$.
Let $p\in \Lambda$ be the fixed point of $F$ in $H_1$. Then
$P=(p,0)$ is a fixed point of $f$ in $\Gamma$. Observe that $P$
does not belong to $V\times B$. Moreover, the intersection of
$W^{u}(P)$ and $V\times B$ contains a disc in the strictly
$f$-invariant family $\mathcal{D}$ of $uu$-dimensional discs.
Thus, $P$ is a distal point of $\Gamma$.
\end{proof}

\section{Historic behavior in a homoclinic class with a blender} \label{S3}

The following result is the main ingredient to get
Theorem~\ref{thm-main}. 
We say a periodic point $Q$ is \emph{homoclinically related} to a periodic point $P$
if $W^{u}(Q)\pitchfork W^{s}(P)\neq \emptyset$ and $W^{u}(P)\pitchfork W^{s}(Q)\neq \emptyset$.
\begin{thm} \label{thm-main-2}
Consider a $C^1$-diffeomorphism $f$ of $M$ having
a homoclinic class $H(P)$, which contains a dynamical $cu$-blender
$\Gamma$ with a distal periodic point in $\Gamma$
homoclinically related to $P$.
Then $H(P)$ has residual historic behavior.
\end{thm}

\begin{proof}
Denote by $\mathcal{D}$ the strictly $f$-invariant family of discs
of the dynamical $cu$-blender $\Gamma$. See Figure \ref{fig3}. By abuse of notation, we
also denote by $\mathcal{D}$ the region in $M$ where the discs of
this family are embedded. Since the homoclinic classes of homoclinically
related saddle points coincide, relabeling if necessary, we will
denote by $P$ the periodic (distal) point of $f$ in $\Gamma$ 
whose distance from $\mathcal{D}$ is larger than some $r>0$. For
simplicity, we assume that $P$ is a fixed point of $f$.

Consider a continuous function $\varphi: M \to  [0,1]$
such that
$\varphi(x)=1$ for all $x$ in the open ball $B_r(P)$ of
radius $r$ centered at $P$ and $\varphi(x)=0$ if $x$ belongs
to $\mathcal{D}$.  Fix $\varepsilon>0$, $N\in\mathbb{N}$ and $x\in
H(P)$. 
We define  $\Delta(x,\varepsilon,N)$ as the set of  points
$y\in B_\varepsilon(x)$ which satisfies the following condition:
there are $n_1, n_2 \geq N$
such that
\begin{equation}\label{eq:1}
    \frac{1}{n_1}\sum_{i=0}^{n_1-1} \varphi(f^i(y))- \frac{1}{n_2}\sum_{i=0}^{n_2-1} \varphi(f^i(y)) >
    \frac{1}{2}.
\end{equation}
The first trivial observation is that $\Delta(x,\varepsilon,N)$ is
open. Moreover, it holds that
\begin{claim} \label{claim}
$\Delta(x,\varepsilon,N) \cap H(P) \not=\emptyset$.
\end{claim}

\begin{figure}[htbp]
\centering
\includegraphics[width=125mm]{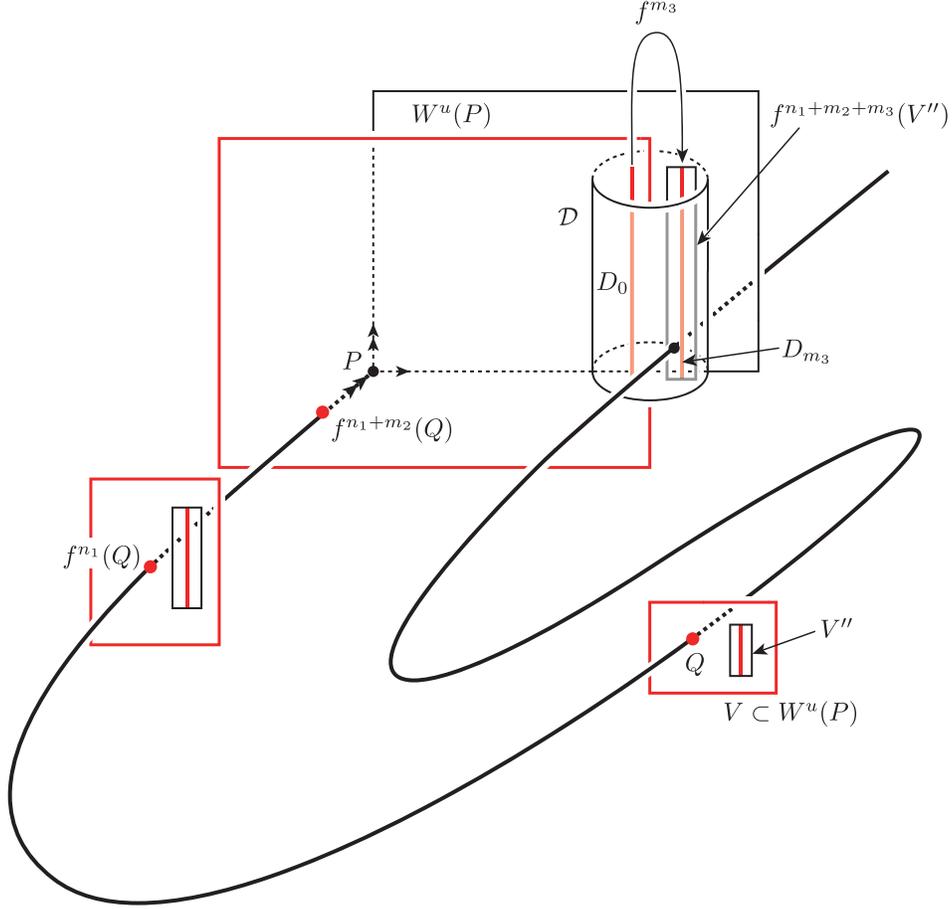}
  \caption{Sketch of the proof of Theorem~\ref{thm-main-2}.}
\label{fig3}
\end{figure}

Let us postpone the proof of this claim to first conclude the
theorem. Take $R_N$ as the union of $\Delta(x,\varepsilon,N)$ for
$\varepsilon>0$ and $x\in H(P)$. Clearly $R_N\cap H(P)$ is an open
and dense set in $H(P)$. Hence the set $R\cap H(P)$, where $R=\cap
R_N$, is a residual set in $H(P)$. Moreover, if $y\in R$ then for
every $N\in \mathbb{N}$, the forward orbit of $y$ has
$N$-conditional historic behavior, i.e., ~\eqref{eq:1} holds.
This implies that the forward orbit of $y$ has historic behavior,
concluding the theorem.

\begin{proof}[Proof of Claim~\ref{claim}]
Since $H(P)$ is a homoclinic class, there exists $Q\in
B_\varepsilon(x)$ belonging to $W^s(P)\pitchfork W^u(P)$. Since
$Q\in W^s(P)$,
there is $m_0\in\mathbb{N}$ such that $f^{m_0}(Q)\in B_r(P)$.
Moreover, assuming $r$ sufficiently small, we  have that
$f^{{m_0}+n}(Q)\in B_r(P)$ for all $n\geq 0$. Take $m_1\in
\mathbb{N}$ such that $n_1=m_0+m_1 \geq N$ and $m_0/n_1 < 1/8$.
Then
\[
     \frac{1}{n_1}\sum_{i=0}^{n_1-1} \varphi(f^i(Q))\geq
     \frac{1}{n_1}\sum_{i=m_{0}}^{n_1-1} \varphi(f^i(Q)) \geq
     \frac{n_1-m_{0}}{n_1}>\frac{7}{8}.
\]
By continuity, we have that
\begin{equation}\label{eq:n1}
\frac{1}{n_1}\sum_{i=0}^{n_1-1} \varphi(f^i(y))>\frac{7}{8} \quad
\text{for all $y\in B_\varepsilon(x)$ close enough to $Q$.}
\end{equation}

Now consider a small disc $V$  in $W^u(P) \cap B_\varepsilon(x)$
transverse to $W^s(P)$ at $Q$ such that $f^{n_1}(V) \subset
B_r(P)$. By the Inclination Lemma, the forward iteration of $V$
approaches $W^u(P)$. Since $P$ is distal, the unstable manifold of
$P$ contains a disc in the open family $\mathcal{D}$. Then we find
$m_2\in \mathbb{N}$ such that $f^{n_1+m_2}(V)$ also contains a
disc $D_0$ in $\mathcal{D}$. See Figure \ref{fig3}.
Moreover, as $\mathcal{D}$ is
strictly $f$-invariant, 
there is a sequence of discs $\{D_n\}_{n\in \mathbb{N}}$ in $\mathcal{D}$
such 
that $D_n \subset f(D_{n-1})$ for all $n\geq 1$. Take $m_3 \in
\mathbb{N}$ such that $(n_1+m_2)/n_2 < 1/8$, where
$n_2=n_1+m_2+m_3$.
Let $V'=f^{-m_3-m_2-n_1}(D_{m_3})\subset V$.
Then,
\[
\frac{1}{n_2}\sum_{i=0}^{n_2-1} \varphi(f^i(y))\leq
     \frac{1}{n_2}\sum_{i=0}^{n_1+m_2} \varphi(f^i(y)) \leq
     \frac{n_1+m_2}{n_2}<\frac{1}{8} \quad \text{for all $y\in  V'$}.
\]
By continuity, we have that
\begin{equation}\label{eq:n2}
\frac{1}{n_2}\sum_{i=0}^{n_2-1} \varphi(f^i(y)) < \frac{1}{8}
\quad \text{for all $y\in B_\varepsilon(x)$ close enough to $V'$}.
\end{equation}
Therefore, Equations~\eqref{eq:n1} and~\eqref{eq:n2}
imply~\eqref{eq:1} concluding that $\Delta(x,\varepsilon,N)$ is
not empty.

To conclude the proof we actually need to show that
$\Delta(x,\varepsilon,N)\cap H(P)\not=\emptyset$. To do this, we
will use that $\Gamma$ is a $cu$-blender. 
Take a small neighborhood $V''$ of $V'$ in 
$V\subset W^u(P)\cap B_\varepsilon(x)$ 
such that $f^{n_2}(V'')$ is a strip foliated by
discs in $\mathcal{D}$. Since $\Gamma$ is a $cu$-blender and $P\in
\Gamma$ is a fixed point, we get that $W^s(P)$ transversally
intersects 
$f^{n_2}(V'')$ at a point $Y$. In particular, $Y$ belongs to
$H(P)$ since $f^{n_2}(V'')\subset W^u(P)$. Thus, $f^{-n_2}(Y)\in
H(P) \cap V'' \subset H(P)\cap \Delta(x,\varepsilon,N)$.
This ends the proof of Claim \ref{claim}.
\end{proof}

Now the proof of Theorem \ref{thm-main-2} is complete.
\end{proof}

\section{Blenders in generic non-hyperbolic homoclinic classes}\label{S4}

We first give a result on historic behavior but with respect
to a \textit{fixed} homoclinic class. This will follow from the existence of blenders and
Theorem \ref{thm-main-2}. Thus, as a part of the proof we review the known arguments for obtaining blenders
in the $C^1$-generic context. Denote by $\mathrm{ind}(P)$ the stable index of a saddle
periodic point $P$ for $f$, and by $P_g$ the continuation of $P$ if $g$ is close to $f$.

\begin{thm} \label{thm1}
Let $f$ be a $C^1$-diffeomorphism of $M$ and 
fix a hyperbolic periodic point $P$ for $f$. 
Assume that the homoclinic class $H(P)$ contains a
hyperbolic saddle $Q$ with
$\mathrm{ind}(Q)=\mathrm{ind}(P)+1$.
Then there is an
open set $\mathcal{U}$ of $C^1$-diffeomorphisms with 
$f \in \overline{\mathcal{U}}$ such that $H(P_g)$ has residual historic
behavior for all $g\in \mathcal{U}$.
\end{thm}

\begin{proof}
Since the homoclinic class of
$P$ must be non-trivial, according to
~\cite[Thm. 1]{BDK12}
we can
approximate $f$ by a diffeomorphism having a $C^1$-robust
heterodimensional cycle  associated with hyperbolic sets containing the
continuations of $P$ and $Q$ (see also ~\cite[Cor.~2.4]{BCDG13}). 

Recall that a diffeomorphism has a \emph{heterodimensional cycle} if
there exists a pair of transitive hyperbolic sets $\Lambda$ and
$\Gamma$ with different indices such that their invariant
manifolds meet cyclically. 
Since
$\mathrm{ind}(P)<
\mathrm{ind}({Q})$,
by construction of the cycle as done in 
\cite[Sect.4]{BD08}, 
there exists a prototypical
$cu$-blender-horseshoe $\Gamma$ homoclinically related with $P$.
The fact that the heterodimensional cycle persists
under perturbations comes from the robustness of the blender. In particular, from
Lemma~\ref{lem-prototipical}, there is a $C^1$-open subset
$\mathcal{V}$ arbitrarily close to $f$ such that any $g\in
\mathcal{V}$ has a dynamical $cu$-blender $\Gamma_g$ and a distal
periodic point in $\Gamma_g$ homoclinically related with $P_g$.
Thus, we are in the assumptions of Theorem~\ref{thm-main-2} and
consequently $H(P_g)$ has residual historic behavior. To conclude,
it is enough to take $\mathcal{U}$ as the union of all of these open
sets $\mathcal{V}$.
\end{proof}
\begin{rem} \label{Gen}
Observe that for a $C^1$-generic diffeomorphism the following facts are known. 
For every pair of hyperbolic periodic points the homoclinic classes either coincide
or are disjoint. And if a
homoclinic class $H(P)$ contains saddles of different indices, then it
also contains periodic points of all intermediate
indices~\cite[Thm. 1 and Lem. 2.1]{ABCDW07}. In particular, for a $C^1$-generic $f$ and a given homoclinic class $H(P)$,
we may assume there exists a point $Q$ with $\mathrm{ind}(Q)=\mathrm{ind}(P)+1$.
Thus, the previous Theorem~\ref{thm1} can be applied in this context.
\end{rem}
In the above theorem the homoclinic class is fixed, but observe that the set of diffeomorphisms obtained is open. However, we
would like to show $C^1$-generic historic behavior for \emph{any} homoclinic
class having index variation (containing saddles of different indices). 
To this end, we present the
following result due to Bonatti and D\'iaz \cite{BD12} on
generation of blenders inside any homoclinic class with certain
index variation.

\begin{thm}[{\cite[Thm.~6.4]{BD12}}] \label{gen}
There is a residual subset $\mathcal{R}$ of $C^1$-diffeomorphisms
$f$ of $M$ such that for every homoclinic class $H(P)$ containing
a hyperbolic saddle $Q$ with $\mathrm{ind}(Q)>\mathrm{ind}(P)$,
there is a transitive hyperbolic set $\Lambda$ containing $P$ and
a $cu$-blender $\Gamma$.
\end{thm}

This theorem follows from standard genericity arguments by
first proving the statement for a fixed homoclinic class
$H(P)$. On the other hand, this is done by means of similar reasoning as in the proof of Theorem~\ref{thm1}
and Remark~\ref{Gen}.
We also would like to emphasize the following: 

\begin{rem} \label{RemProt}
The $cu$-blenders constructed in Theorem~\ref{gen} 
come from the perturbations of prototypical
$cu$-blender horseshoes in the sense of Section~\ref{prototypical}, and in particular are dynamical blenders.
\end{rem}

Now we prove our main Theorem.

\begin{proof}[Proof of Theorem~\ref{thm-main}]
Consider the residual set $\mathcal{R}$ given in
Theorem~\ref{gen}. We can assume that for any $f\in \mathcal{R}$
and every pair of hyperbolic periodic points $P$ and $Q$ of $f$
either $H(P)=H(Q)$ or $H(P)\cap H(Q)=\emptyset$
(see~\cite[Lem. 2.1]{ABCDW07}). Now, fix $f\in \mathcal{R}$ and let $P$ be a
saddle periodic point of $f$ whose homoclinic class contains a
saddle $Q$ of different stable index.

If $\mathrm{ind}(Q)>\mathrm{ind}(P)$ then by Theorem~\ref{gen}
we get 
a $cu$-blender
$\Gamma$ in $H(P)$,
whose saddles are homoclinically related with $P$. By Remark~\ref{RemProt}, 
$\Gamma$ is a dynamical $cu$-blender and from
Lemma~\ref{lem-prototipical} has a distal periodic point related
with $P$. Thus, we are in the assumptions of
Theorem~\ref{thm-main-2} and consequently $H(P)$ has residual
historic behavior.

Next we suppose that 
$\mathrm{ind}(Q)<\mathrm{ind}(P)$. 
Since $H(P)=H(Q)$, 
$H(Q)$ contains points of different indices and 
similarly as above we
can apply Theorem~\ref{gen} for this homoclinic class. 
Hence, we
get that $H(Q)$ has residual historic behavior. As $H(Q)=H(P)$ we
also get the same conclusion for $H(P)$ and conclude the proof.
\end{proof}

%




\end{document}